\documentclass[a4paper,10pt]{article}%

\usepackage[utf8]{inputenc}%
\usepackage[T1]{fontenc}%
\usepackage[english]{babel}%

\usepackage{lmodern}%
\usepackage[final]{graphicx}%
\usepackage{color}%
\usepackage{amsmath,amssymb,amsthm,dsfont}
\usepackage{hyperref}%
\usepackage{enumerate}%
\usepackage[labelfont=bf,labelsep=period]{caption}%

\allowdisplaybreaks%
\newcommand{\Nbb}{\mathbb{N}}
\newcommand{\Zbb}{\mathbb{Z}}

\newcommand{\Sc}{\mathcal{S}}

\newcommand{\xb}{\mathbf{x}}%
\newcommand{\Xb}{\mathbf{X}}%
\newcommand{\yb}{\mathbf{y}}%
\newcommand{\Yb}{\mathbf{Y}}%

\newcommand{\oh}{\frac{1}{2}}%
\newcommand{\nullc}{\mathbf{0}}%
\newcommand{\lrm}{\mathrm{left}}%
\newcommand{\rrm}{\mathrm{right}}%
\newcommand{\bll}{\bullet}%

\newcommand{\emp}[1]{\emph{#1}}%

\newcommand{\Exp}{\operatorname{\mathbf{E}}}%
\newcommand{\Prob}{\operatorname{\mathbf{P}}}%
\newcommand{\ind}{\operatorname{\mathds{1}}}%

\numberwithin{equation}{section}%

\theoremstyle{plain}%
\newtheorem{theorem}{Theorem}%
\newtheorem{corollary}{Corollary}%
\newtheorem{prop}{Proposition}%
\newtheorem{lemma}{Lemma}%

\theoremstyle{definition}%
\newtheorem{definition}{Definition}%
\newtheorem{remark}{Remark}%

\newcommand{\qqed}{\hfill$\blacksquare$}%

\renewenvironment{proof}[1][Proof]{\begin{trivlist}%
\item[\hskip \labelsep {\bfseries #1}]}{\qqed\end{trivlist}}%

\newcommand{\nalaffl}{Institute of Mathematics, Budapest University of Technology and Economics, Egry J. u. 1., Budapest, H-1111, Hungary.\ }%

\newcommand{\titl}{
Phase transition in a double branching annihilating random walk
}%

\newcommand{\auths}{
Attila L\'aszl\'o Nagy\,\footnote{\nalaffl}\\[0.2em]
{\normalsize \url{attilalaszlo.nagy@gmail.com}}
}%

\newcommand{\dat}{\today}%

\title{\titl}%
\author{\auths}%
\date{\dat}%

\begin{document}
\maketitle

\begin{abstract}
This paper investigates the long-time behavior of double branching annihilating random walkers with nearest-neighbor dependent rates. The system consists of even number of particles which can execute nearest-neighbor random walk and they can as well give birth in a parity conserving manner to two other particles with rates $1$ and $b$, respectively, until they meet. Upon meeting, each of the adjacent particles can branch with rate $p\cdot b$ while it can annihilate, i.e.\ hop on, the other particle with rate $p$ for some $0< p\leq 1$. This process first appeared in \cite{baw} and can be considered as the extension of that of \cite{sudburybap}. We prove that in some region of the parameters $(p,b)$, the process survives with positive probability. Combining the extinction result of \cite{sudburybap} it shows a phase transition phenomenon for this model. In some sense our result also shows the sharpness of the assumptions of \cite{dbarw}. We use similar arguments that was developed by M.\ Bramson and L.\ Gray in \cite{bramsongray}.
\end{abstract}

\noindent \textbf{Keywords.} Non-attractive particle system, double branching annihilating random walk, parity conserving, dependent rate, survival, phase transition.

\bigskip

\noindent \textbf{Acknowledgement.} A.\ L.\ N.\ thanks M\'{a}rton Bal\'{a}zs for his careful review and many helpful suggestions to the manuscript. The author also thanks B\'{a}lint T\'{o}th and M\'{a}rton Bal\'{a}zs for general discussions related to (non-attractive) interacting particle systems.

\section{Introduction}%

The system of double branching annihilating random walkers consists of a finite population of particles interacting with each other on the lattice $\Zbb$, which can execute:
\begin{description}
\item[(RW)] nearest-neighbor random walk; or
\item[(BR)] branching, i.e.\ a particle can give birth to two offsprings placing them to the two neighboring lattice points.
\end{description}
To keep each occupation number 0-1, two particles occupying the same position are simultaneously \emp{annihilated}. This can happen both for the (RW) and the (BR) steps. Note that this dynamics conserves the \emp{parity} of the total particle number. We assume throughout the article that \emp{even} number of particles are present initially.

One can look at $\Yb$ as an interface (or boundary) process of another model $\Xb$, which is called as the \emp{swapping voter model} (see \cite{votermodels}). We define $\Xb$ on the \emp{half-integer lattice} $\Zbb+\oh$ to be of $\Yb$ as the magnitude of its discrete spatial gradient. For instance $\Xb=\delta_{\oh}$ if and only if $\Yb=\delta_0+\delta_1$, where $\delta$ is the Kronecker symbol ($\delta_i(j)=1$ if $j=i$ otherwise $0$). To avoid ambiguity we refer to the coordinates of $\Xb$ as \emp{heights}.
Now, the above described dynamics of $\Yb$ can be easily translated into the language of $\Xb$ as:
\begin{description}
\item[(Flip)] (corresponding to \textbf{(RW)})\hfill\\
A height can switch to the value of one of its neighboring heights.
\item[(Exclusion)] (corresponding to \textbf{(BR)})\hfill\\
An adjacent zero-one (one-zero) pair of heights can swap values.
\end{description}
When only the \emp{spin-flip} interactions (Flip) are present the process is in the range of \emp{voter models} (\cite[Part II.]{liggettbiblestoch}). With only the (Exclusion) steps we arrive to the well-known \emp{exclusion processes} (\cite[Part III.]{liggettbiblestoch}). Both of these were \emp{separately} under extensive studies in the past decades. Much less known, however, about those processes in which the above two interactions are \emp{jointly} present. This mixed dynamics makes our process rather interesting and highly non-trivial.

Both processes ($\Xb$ and $\Yb$) are in the class of additive and cancellative interacting particle systems (see \cite{griffeath}). In these sort of systems a certain monotonicity property is violated making them \emp{non-attractive} (see \cite[pp.\ 71--72, pp.\ 380--384]{liggettbible}). This lack of monotonicity practically prevents us to compare these processes via domination arguments. Hence the long-time behavior such as the survival or extinction seem to be difficult to treat in general (see \cite{bramsongray,sudburybap,bramsondurrett,sudburysurvival,blathcoex,rebvoter,blathkurt,extwindow} and further references therein).

\bigskip
\noindent\emp{Earlier results}. The model of double branching annihilating random walkers was introduced by A.\ W.\ Sudbury \cite{sudburybap}. In that article it was shown that the process with constant rate hopping and branching dies out a.s.\ in finite time. Shortly after this, \cite{baw} introduced a double branching annihilating random walk with nearest-neighbor dependent rates which is an extension of Sudbury's process. In this latter case particles perform unit rate hopping and rate $b>0$ branching \emp{until} two of them meet. If a particle is about to hop or branch that effects at least one particle in its neighborhood the corresponding movement takes place with (reduced) rates $p$ or $p\cdot b$, respectively, for some $0< p\leq 1$. Via computer simulations \cite{baw, zhong} demonstrated that this model undergoes phase transition, i.e.\ survival-extinction regimes exist and critical values were numerically determined. Recently, J.\ Blath and N.\ Kurt \cite{blathkurt} considered specific double branching annihilating random walk dynamics for which phase transition was proved. In these latter models, however, the configuration dependent branching can take place for possibly non-nearest neighbor lattice points as well.

Double branching annihilating random walk belongs to the \emp{parity conserving} (PC) class of particle systems which has been in the scope of relevant studies in the Physics literature as well (\cite{cardytaubertheory, cardytauberfield}, for a monograph see \cite[Sec.\ 4.6, pp.\ 117--122]{odor} and further references therein).

\bigskip
\noindent\emp{Results of this paper}. We adopt the setup of \cite{baw} and we rigorously prove that the double branching annihilating random walk with the above nearest-neighbor dependent interaction survives in some region of the rate parameters $(p,b)$ with positive probability. Together with \cite{sudburybap} it then proves a phase transition phenomenon for this model. To the best of our knowledge this is the first time that phase transition is rigorously showed for a nearest-neighbor dependent non-attractive particle system in the PC class. We follow the approach that was developed by M.\ Bramson and L.\ Gray in \cite{bramsongray} for a \emp{parity violating} branching annihilating random walk. We highlight that the survival of this particular model with nearest-neighbor dependent rates in some way shows the sharpness of the assumptions of \cite{dbarw}. In \cite{dbarw} ergodicity, as seen from the left-most particle position, was proved under quite general conditions for the rate functions provided that \emp{odd} number of particles are present initially.

\bigskip
\noindent\emp{Organization of the paper}. We precisely define the processes of interest in Section \ref{sec:models}. Section \ref{sec:results} details the main results while the proofs lie in Section \ref{sec:proofs}.

\section{The models}\label{sec:models}

\paragraph{Double branching annihilating random walk.} Define the configuration space
\begin{equation}\label{eq:configspace}
\Sc = \big\{\yb\in\{0,1\}^{\Zbb}\,:\,\textstyle\sum_{i\in\Zbb}y_i\text{ is even}\big\},
\end{equation}
that is a configuration $\yb=(y_i)_{i\in\Zbb}\in\Sc$ consists of holes and (a finite population of) particles. For a lattice point $i\in\Zbb$ we interpret $y_i=1$ as the presence of a particle while $y_i=0$ means the absence of such a particle. Now, the \emp{double branching annihilating random walk}
\[
\Yb=(\Yb(t))_{t\geq0}=(\ldots,Y_{i-1}(t), Y_i(t), Y_{i+1}(t),\ldots)_{t\geq0}
\]
is a continuous-time Markov process on $\Sc$ which allows the following sort of transitions
\begin{equation}
\begin{aligned}\label{eq:transitions}
\yb\;\overset{r_{i}(\yb)}{\longrightarrow}\;\yb-\delta_i+\delta_{i+1} &\qquad \text{(left jump),}\\
\yb\;\overset{\ell_{i}(\yb)}{\longrightarrow}\;\yb-\delta_{i}+\delta_{i-1} &\qquad \text{(right jump),}\\
\yb\;\overset{b_{i}(\yb)}{\longrightarrow}\;\yb+\delta_{i-1}+\delta_{i+1} &\qquad \text{(branching)}
\end{aligned}
\end{equation}
which, conditioned on $\Yb(t)=\yb\in\Sc$, take place independently of each other with instantaneous rates $r_i(\yb)$, $\ell_i(\yb)$ and $b_i(\yb)$ for $i\in\Zbb$, respectively. We assume that $r_i(\yb)=\ell_i(\yb)=b_i(\yb)=0$ whenever $y_i=0$, that is only particle can perform actions. Note that all the operations of \eqref{eq:transitions} are meant modulo $2$. In the following the initial configuration $\Yb(0)=\Yb_0$ is always chosen to be a deterministic element of $\Sc$.

We remark that \cite{dbarw} dealt with an equivalent formulation of the above dynamical rules on the state space $\{-1,0,1\}^\Zbb$ with positive and negative particles alternating from left to right. For sake of simplicity we have insisted on defining the above process on $\{0, 1\}^{\Zbb}$ which somewhat shortcuts the notations of the present article. However, all the results below would hold in either formulation.

The functions $r_{\bll}(\cdot)$, $\ell_{\bll}(\cdot)$ and $b_{\bll}(\cdot)$ correspond to the nearest-neighbor right, left jumping and branching, respectively. Throughout the article we make the following choices:
\begin{gather}
r_i(\yb) = y_i\cdot(1 - (1-p)\cdot y_{i+1}),\qquad \ell_i(\yb) = y_i\cdot(1 - (1-p)\cdot y_{i-1}),\label{eq:rwrates}\\
b_i(\yb) = b\cdot y_i\cdot(p + (1-p)\cdot(1-y_{i+1})\cdot(1-y_{i-1})),\label{eq:branchrates}
\end{gather}
where $0<b$, $0< p\leq 1$ are fixed parameters and $\yb\in\Sc$. In plain words, particles execute right, left hopping with unit rate until they meet when they can annihilate each other with rate $2p$. On the other hand branching occurs with rate $b$ if both sides of the branching particle are empty, otherwise it will take place with rate $p\cdot b$. Note that the dynamics keeps the configuration space \eqref{eq:configspace} invariant. It also easily follows that $\Yb$ with the above rates has a.s.\ finite number of particles at each time (see for e.g.\ \cite[Ch. I.]{liggettbible}).

Some further notations will ensue. We denote the \emp{empty configuration} by $\nullc$ in which no particles are present (absorbing/vacuum state). Furthermore, let $i_{\lrm}=\min\{i\in\Zbb:y_i\neq 0\}$ and $i_{\rrm}=\max\{i\in\Zbb:y_i\neq 0\}$ be the leftmost and rightmost particle position of $\nullc\neq \yb\in\Sc$, respectively. We define the \emp{width} $w$ of $\yb\in\Sc$ by letting $w(\yb)=i_{\rrm} - i_{\lrm} + 1$ with the convention that $w(\nullc)=0$. Note that $w$ cannot take on the value $1$. For sake of simplicity the width process of $(\Yb(t))_{t\geq0}$ is shortened to $W(t)=w(\Yb(t))$ for $t\geq 0$. With a slight abuse of notation we denote $i_{\lrm}(t)$ and $i_{\rrm}(t)$ by the left and right end particle position, respectively, as long as $\Yb(t)\neq\nullc$.
\paragraph{Swapping voter model.} For a configuration $\yb\in\Sc$ we define the height $\xb\in\{0,1\}^{\Zbb+\oh}$ to obey the relation
\begin{equation*}
y_i = |x_{i+\oh} - x_{i-\oh}| \qquad (i\in\Zbb).
\end{equation*}
That is an adjacent 01 (10) pair in $\xb$ will result in a single particle of $\yb$ or in other words $\yb$ marks the phase boundaries of $\xb$.
\begin{figure}[htbp]
\centering
\includegraphics[scale=0.23]{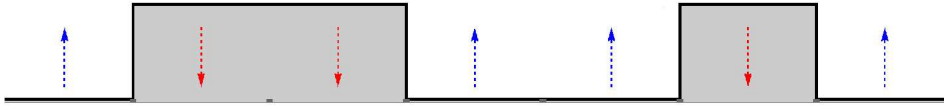}%
\caption[Transitions of double branching annihilating random walkers with heights]{A four particle configuration with heights. A \emp{single arrow} corresponds to a flip ({\color{blue}blue}: from 0 to 1, {\color{red}red}: from 1 to 0). From left to right these can happen with rates $1,1,1,1,1,2p$ and $1$, respectively. Adjacent \emp{opposite arrows} can produce a swap in heights which can happen with rates $b,b,p\cdot b$ and $p\cdot b$, respectively.}
\label{fig:dbarw-model}
\end{figure}

Now, we define the \emp{swapping voter model} $(\Xb(t))_{t\geq0}$ from $(\Yb(t))_{t\geq0}$ as $\xb$ was defined above from $\yb$. The transitions of \eqref{eq:transitions}, conditioned on $\Xb(t)=\xb$ for $t\geq0$, transform to the following rules: a bit of $\xb$ can change from 0 to 1 or from 1 to 0 (\emp{spin-flip}, that is hopping in $\Yb$); while adjacent bits can exchange values, that is from 01 to 10 or from 10 to 01 (\emp{exclusion}, that is branching in $\Yb$). Figure \ref{fig:dbarw-model} exhibits a particular configuration of $\Xb$ with transitions and rates.

\section{Main results}\label{sec:results}%

Our main assertion concerns the survival of double branching annihilating random walkers depending on the choice of the parameters $b$ and $p$. In particular
\begin{theorem}\label{thm:survival}
For every $b>10^4$ and $p<2b^{-2}$ there exists a $c=c(\Yb_0)$ such that
\begin{equation*}
0 < c \leq \Prob\{\Yb(t)\neq \nullc \text{ {\rm for all }}t>0\,|\,\Yb(0)=\Yb_0\}
\end{equation*}
holds for all $\nullc\neq\Yb_0\in\Sc$. That is with positive probability the particles of $\Yb$ will \emp{survive}.
\end{theorem}
\begin{corollary}\label{cor:svmsurv}
Under the assumptions of Theorem \ref{thm:survival}, there exists a $\tilde{c}=\tilde{c}(\Xb_0)$ such that
\begin{equation*}
0 < \tilde{c} \leq \Prob\{\Xb(t)\neq \nullc\text{ for all $t>0$}\,|\,\Xb(0)=\Xb_0\}
\end{equation*}
holds for all $\nullc\neq\Xb_0\in\Sc$.
That is the zero-one heights of $\Xb$ \emp{coexist} with positive probability.
\end{corollary}
\begin{remark}
On the heuristic level the above results seem to be quite straightforward since two adjacent particles are likely to \emp{jump away} rather than to annihilate each other if $p<1$. However, the non-attractive property of $\Yb$ makes this reasoning highly non-trivial.
\end{remark}
\begin{remark}
Theorem \ref{thm:survival} implies that the width of $\Yb$ unboundedly grows in time assuming survival. This is not immediate (but is expected) for the total particle number of $\Yb$ due to the parity conserving property.
\end{remark}
\begin{remark}
For each fixed $b>10^4$ the model exhibits \emp{phase transition} as $p$ is tuned from $0$ to $1$. In particular for $p=1$ the process dies out a.s.\ in finite time (see \cite{sudburybap}), while for $p<2b^{-2}$ it survives with positive probability (see above, Theorem \ref{thm:survival}). Articles \cite{baw, zhong} numerically showed that the critical value is strictly less than $1$. Currently, there are no rigorous methods available for determining the exact critical value.
\end{remark}
\begin{remark}
We remark that the hopping rates of \eqref{eq:rwrates}--\eqref{eq:branchrates} do \emp{not} satisfy the \textbf{A4} assumption of \cite{dbarw}, where, under quite general conditions, the ergodicity of $\Yb$ as seen from the leftmost particle position (i.e.\ the interface tightness of $\Xb$) was proven. (Note that \cite{dbarw} assumes \emp{odd} number of initial particles.) Hence, Theorem \ref{thm:survival} shows the \emp{sharpness} of the assumptions of \cite{dbarw} provided that interface tightness of $\Xb$ implies extinction for $\Yb$ in this model (see also \cite[Question Q2. on pp. 64]{votermodels}).
\end{remark}

\section{Proofs}\label{sec:proofs}%

We divide the proof of Theorem \ref{thm:survival} into two subsections. In the first one we deduce an estimate for the growth of the width which we will heavily use in the second part to conclude the survival of $\Yb$.

\subsection{Growth of the width process}%

One of the key points of survival is a domination argument based on a lower estimate of the width process $W$. The difficulty arising from the analysis of $W$ is that it can decrease by more than $2$ units at sometimes but can only grow at most by one unit at any transition. Along with some additional conditions on the rates we will show that this sort of ``large decreases in width'' happens rarely.

In the first step we establish a lower estimate to $W$ until it does not decrease by more than two units. Let $0<\tau_1,\tau_2,\ldots$ be the transition times of $(W(t))_{t\geq0}$, $\tau_0\equiv0$ and define
\[
N=\inf\{n\in\Nbb:W(\tau_n)=0\text{ or } W(\tau_{n})-W(\tau_{n-1})<-2\}.
\]
In plain words the process $\Yb$ dies out or its width first decreases by more than two units in the $N^{\text{th}}$ step.
Next, let $Z_1,Z_2,\ldots$ be a sequence of independent random variables which are also independent of the process $\Yb$ having common distribution:
\begin{align*}
\Prob\{Z_n=-2\} = \Prob\{Z_n=-1\} & = \frac{2}{2+b}\\
\Prob\{Z_n=+1\} &= 1 - \frac{4}{2+b}
\end{align*}
for all $n\in\Zbb^+$ provided that $b>2$. We also define the \emp{embedded growth process} of $W$ being
\begin{equation}\label{eq:widthcoupling}
V_n=
\begin{cases}
W(\tau_n)-W(\tau_{n-1})& \text{if $n<N$};\\
Z_n& \text{otherwise}.
\end{cases}
\end{equation}
We are ready to state the following assertion concerning the width process.
\begin{prop}\label{prop:widthgrowth}
For all $n\in\Zbb^+$: $V_1+V_2+\cdots+V_n$ \emp{stochastically dominates} $Z_1+Z_2+\cdots+Z_n$ with $\Exp Z_n > 0$ provided that $b>8$ and $p < 2b^{-2}$ hold.
\end{prop}
\begin{proof}[\textbf{Proof of Proposition \ref{prop:widthgrowth}.}]
Since the range of the corresponding random variables only contains three distinct elements ($\{-2,-1,1\}$), it is enough to prove that
\begin{align*}
\Prob\{V_n=-2\,|\,V_1,\ldots,V_{n-1}\}&\leq \Prob\{Z_n=-2\}\\
\Prob\{V_n=-1\,|\,V_1,\ldots,V_{n-1}\}&\leq \Prob\{Z_n=-1\}
\end{align*}
hold for all $n\in\Zbb^+$ whenever $b>8$ and $p < 2b^{-2}$. Indeed, we show that
\begin{align}
\Prob\{V_n=-2\,|\,\Yb(\tau_{n-1}), n<N\}&\leq \frac{2}{2+b},\label{eq:decrm2}\\
\Prob\{V_n=-1\,|\,\Yb(\tau_{n-1}), n<N\}&\leq \frac{2}{2+b}\label{eq:decrm1}
\end{align}
from which the domination easily follows using the strong Markov property of $W$ for the stopping time $N$ since $V_n=Z_n$ for all $n\geq N$.

For \eqref{eq:decrm2} we determine the maximum rate at which the width can decrease by at least two units, and the minimum rate it can increase by one. The former one is clearly $4p$ by a jumping annihilation at either side of $W$, while the latter is $2+2pb$ by a jumping or a branching of the first and last particles of $\Yb$. Putting these together we get the upper bound of \eqref{eq:decrm2}:
\[
\Prob\{V_n=-2\,|\,\Yb(\tau_{n-1}), n<N\}\leq \frac{4p}{2+4p+2p\cdot b} < \frac{2}{2+b}.
\]

In case of \eqref{eq:decrm1}, we pretty much copy the previous idea but this time we need to take into account the possible occupations at lattice points $i_{\lrm}+1$ and $i_{\rrm}-1$. So let $E$ be the event that the change in the width at time $\tau_n$ occurs at the left end. Since the rate at which the left end can decrease under $E$ is $1+b$ while it can decrease with unit rate, it follows that
\[
\Prob\{V_n=-1\,|\,\Yb(\tau_{n-1}),n<N,E,Y_{i_{\lrm}(\tau_{n}^-)+1}=0\}\leq \frac{1}{2+b}.
\]
Similarly
\[
\Prob\{V_n=-1\,|\,\Yb(\tau_{n-1}),n<N,E,Y_{i_{\lrm}(\tau_{n}^-)+1}=1\}\leq \frac{p\cdot b}{1+2p\cdot b} < \frac{2}{2+b}
\]
using the bound $p < 2b^{-2}$. By symmetry reasons the same bounds would also hold when the change in the width happens at the right end particle. So we end up with $0 < \Exp Z_n = 1 - \frac{10}{2+b}$ when $b>8$.
\end{proof}

\subsection{Survival via separation times}%

Closely following \cite{bramsongray} we first fix some notions which will turn out to be crucial for the subsequent argument.
\begin{definition}[$i$-gap, $i$-transition]
We say that an $i$-gap occurs at point $i\in\Zbb^+$ at some time $t\geq0$ if
\begin{enumerate}[(i)]
\item $Y_{i_{\lrm}(t) + i - 1}(t) = 1$, $Y_{i_{\lrm} + i}(t) = 0$ and
\item either $Y_{i_{\lrm} + i + 1}(t) = 1$ or $Y_{i_{\lrm} + i + 2}(t) = 1$ holds.
\end{enumerate}
A transition that produces an $i$-gap will be called an $i$-transition.
\end{definition}
\begin{definition}[separation time]
We say that at time $\sigma\geq 0$ the process is \emp{permanently separated} by a gap if
\begin{enumerate}[(i)]
\item for some $i\in\Zbb^+$, $i<w(\Yb(\sigma))$: $Y_{i_{\lrm}(\sigma)+i}(\sigma)=0$; and
\item for all $t>\sigma$, the particles of $(Y_j(\sigma))_{j<i}$ and all their subsequent descendants do \emp{not} meet, i.e.\ become adjacent, the remaining particles that is with $(Y_j(\sigma))_{j > i}$ and none of their descendants.
\end{enumerate}
For sake of brevity we also say that $\sigma$ is a \emp{separation time}.
\end{definition}
The last assumption of the previous definition implies that at the separation time $\sigma$ the process $\Yb$ falls apart into two sub-processes being separated by at least one vacant site for every subsequent time. Notice that by its definition $\sigma$ is \emp{not} a Markovian stopping time and so we need to be careful restarting the process from $\sigma$.

The number of \emp{positive separation times} is denoted by $K$. Furthermore, let $A$ be the event:
\[
A = \big\{\text{$0$ is \emp{not} a separation time and there exists a $t^*>0$ s.\ t. } \Yb(t^*)=\nullc\big\}.
\]
It is worth noting that under $A$, a separation time $\sigma$ is always produced by an $i$-transition for some $i\in\Zbb^+$.
Concerning $A$ we will prove the following
\begin{lemma}\label{lem:separation}
For all $k\in\Zbb$ and $\nullc\neq\Yb_0\in\Sc$ we have
\begin{equation}\label{eq:sepbound}
\Prob\{A\text{ {\rm and }}K\leq k\,|\,\Yb(0)=\Yb_0\} < \frac{1}{20b} < 10^{-5}
\end{equation}
provided that $b>10^4$ and $p<2b^{-2}$.
\end{lemma}
\begin{proof}[\textbf{Proof of Theorem \ref{thm:survival}.}]
By letting $k\to +\infty$ the previous Lemma \ref{lem:separation} implies that
\begin{equation}\label{eq:lemmaimpliedsurv}
\Prob\{A\,|\,\Yb(0)=\Yb_0\}=\Prob\{A, K<+\infty\,|\,\Yb(0)=\Yb_0\}<1.
\end{equation}
If there are holes between the particles of $\Yb_0$ then $0$ can possibly be a separation time. To circumvent this difficulty we start the process $\Yb$ from $\Yb_0$ and stop at time $1$. Let $u_1 = \Prob\{\Yb(1)=\delta_0+\delta_1\,|\,\Yb(0)=\Yb_0\}$. Since the dynamics allows $\Yb$ to access from any configuration (besides $\nullc$) to any other configuration with positive probability in finite time, it easily follows that $0<u_1<1$. Then we can write that
\begin{align*}
\Prob\{\text{$\Yb$ dies out}\,|\,\Yb(0)=\Yb_0\}
&=\Prob\{\text{$\Yb$ dies out}, K<+\infty\,|\,\Yb(0)=\Yb_0\}\\
&=\Prob\{\text{$\Yb$ dies out}, K<+\infty\,|\,\Yb(1)=\delta_0+\delta_1\}\cdot u_1\\
&+\Prob\{\text{$\Yb$ dies out}, K<+\infty\,|\,\Yb(1)\neq \delta_0+\delta_1\}\cdot(1 - u_1)\\
&\leq 1 + (\Prob\{A, K<+\infty\,|\,\Yb(0)=\delta_0+\delta_1\}-1)\cdot u_1,
\end{align*}
which is strictly less than $1$ by \eqref{eq:lemmaimpliedsurv}. At the last inequality of the previous display we indeed restarted the process from $\delta_0+\delta_1$. The proof is then complete.
\end{proof}
\begin{proof}[\textbf{Proof of Corollary \ref{cor:svmsurv}.}]
It directly comes from the bijection we described in the second part of Section \ref{sec:models}.
\end{proof}
\begin{proof}[\textbf{Proof of Lemma \ref{lem:separation}.}]
By induction on $k$, we are going to prove that with $C(b)=1/(20b)$
\begin{equation}\label{eq:inductionbound}
\Prob\{A\text{ and }K\leq k\,|\,\Yb(0)=\Yb_0\} \leq C(b)\cdot\frac{2^{-w(\Yb_0)}}{w(\Yb_0)^2}
\end{equation}
holds for all $b>10^4$ and $p<2b^{-2}$, which implies \eqref{eq:sepbound}.
\paragraph{Case $k=0$.} In this case there are no separation times at all which implies that the width process $(W(t))_{t\geq0}$ never decreases by more than $2$ units. Now, let $M$ be the number of times the process $(w(\Yb_0)+V_1+\cdots+V_n)_{n\in\Zbb^+}$ hits the value $2$. Recall the definition of $V$'s from \eqref{eq:widthcoupling}.
For the process $\Yb$ to die out it first needs to reach a configuration in which two neighboring particles are only present and then a jumping annihilation (with probability $\frac{p}{1+p+p\cdot b}$) results in the absorbing state $\nullc$. Let $\kappa_n$ be the $n^{\text{th}}$ random index for which $(W(\tau_l))_{l\in\Zbb^+_0}$ reaches $2$. We define $\kappa_n$ to be $+\infty$ in case of the process $\Yb$ would die out before $W$ reached $2$ for the $n^{\text{th}}$ time.
It follows that
\begin{multline*}
\Prob\{A \text{ and } K=0\,|\,\Yb(0)=\Yb_0\}\\
\leq \sum_{n=1}^{+\infty}\Prob\{\Yb(\tau_{\kappa_n+1})=\nullc,\kappa_n<+\infty, K=0\,|\,\Yb(0)=\Yb_0\}\\
\leq \sum_{n=1}^{+\infty}\Prob\{\Yb(\tau_{\kappa_n+1})=\nullc, \kappa_n<+\infty, M\geq n\,|\,\Yb(0)=\Yb_0\}\\
\leq \sum_{n=1}^{+\infty}
\Prob\{\Yb(\tau_{\kappa_n+1})=\nullc\,|\,\kappa_n<+\infty, M\geq n, \Yb(0)=\Yb_0\}\\
\cdot \Prob\{M\geq n\,|\, \Yb(0)=\Yb_0\}\\
\leq \frac{p}{1+p+p\cdot b}\cdot\sum_{n=1}^{+\infty}\Prob\{M\geq n\,|\,\Yb(0)=\Yb_0\}
\end{multline*}
Pick a $v>1$ and define $\gamma(v,b)=\Exp v^{-Z_1}=\frac{1}{2+b}\cdot(2v^2 + 2v + (b-2)/v)$. It is not hard to see that for each $b>10^4$ there exists an $v^*=v^*(b) \geq 4$ such that $\gamma(v^*,b)=\oh$. In the following we use this setting.
\begin{align}
&\Prob\{M\geq n\,|\,\Yb(0)=\Yb_0\}\nonumber\\
&\quad\leq \sum_{l=n}^{+\infty}\Prob\{V_1+V_2+\cdots+V_l \leq 2-w(\Yb_0)\,|\,\Yb(0)=\Yb_0\}\nonumber\\
&\quad\leq \sum_{l=n}^{+\infty}\Prob\{Z_1+Z_2+\cdots+Z_l\leq 2-w(\Yb_0)\,|\,\Yb(0)=\Yb_0\}\nonumber\\
&\quad\leq \sum_{l=n}^{+\infty}\Prob\{(v^*)^{-Z_1-Z_2-\cdots-Z_l} \geq (v^*)^{w(\Yb_0)-2}\,|\,\Yb(0)=\Yb_0\}\nonumber\\
&\quad\leq \bigg(\frac{1}{v^*}\bigg)^{w(\Yb_0)-2}\cdot\sum_{l=n}^{+\infty}\gamma(v^*,b)^l \leq 4^{2-w(\Yb_0)}\cdot 2^{1-n}\label{eq:largedevestim}
\end{align}
provided that $b>10^4$ and taking advantage of the stochastic domination of Proposition \ref{prop:widthgrowth}. Putting the above together and using that $p<2b^{-2}$ we arrive to the estimate
\begin{align}
\Prob\{A\text{ and }K=0\,|\,\Yb(0)=\Yb_0\}
&\leq \frac{64}{2+b^2} \cdot 4^{-w(\Yb_0)}\nonumber\\
&\leq \frac{64}{2+b^2} \cdot \frac{2^{-w(\Yb_0)}}{w(\Yb_0)^2} \cdot \max_{x\geq 0}(x^2\cdot 2^{-x})\nonumber\\
&\leq \frac{128}{2+b^2} \cdot \frac{2^{-w(\Yb_0)}}{w(\Yb_0)^2}.\label{eq:zerobound}
\end{align}
This implies \eqref{eq:inductionbound} in case of $k=0$.
\paragraph{Inductive step.}
We let $k=m$ for a fixed $m>1$ and we assume that \eqref{eq:inductionbound} holds for all $0\leq k\leq m-1$.
Considering \eqref{eq:zerobound} it is enough to prove that
\begin{equation*}
\Prob\{A\text{ and \textbf{1}}\leq K\leq m\,|\,\Yb(0)=\Yb_0\} \leq \bigg(C(b)-\frac{128}{2+b^2}\bigg)\cdot \frac{2^{-w(\Yb_0)}}{w(\Yb_0)^2},
\end{equation*}
The approach we follow is close to that of \cite{bramsongray}. Under the event $1\leq K\leq m$, there is at least one positive separation time. Hence it is reasonable to call the very first such time $\sigma$. At $\sigma$ an $i$-gap (for some $i\in\Zbb$) separates the whole process $\Yb$ into two sub-processes $(\Yb^{(1)}(\sigma+t))_{t\geq0}$ and $(\Yb^{(2)}(\sigma+t))_{t\geq0}$ which remain separated by at least one vacant site for all future time. Under $A$ these sub-processes have strictly less than $m$ separation times and so we can apply the induction hypothesis. We note that for each of these sub-processes, $\sigma$ cannot be a separation time separately since $\sigma$ is the very first separation time of $\Yb$.

For this argument to be accomplished we need to stop $\Yb$ at $\sigma$ and then restart the sub-processes from that time. Notice, however, that $\sigma$ is \emp{not} a Markovian stopping time since it depends on the whole future of $\Yb$ and so the above idea cannot be directly applied with $\sigma$. To circumvent this difficulty we define the Markovian sequence of stopping times $\sigma_{1}^i,\sigma_{2}^i,\ldots,\sigma_{n}^i,\ldots$ for each $i\in\Zbb^+$, where $\sigma_n^{i}$ denotes the $n^{\text{th}}$ step at which an $i$-gap occurred ($n\in\Zbb^{+}_0$). Recall that an $i$-gap is \emp{not} necessarily separating but for a separation time to occur we need an $i$-transition to produce that for some $i\in\Zbb^+$.

Using the law of total probability we obtain that
\begin{align}
&\Prob\{A\text{ and }1\leq K\leq m\,|\,\Yb(0)=\Yb_0\}\nonumber\\
&\quad=\sum_{j=3}^{+\infty}\sum_{i=1}^{j-1}\sum_{n=1}^{+\infty}\Prob\{A, 1\leq K\leq m, W(\sigma)=j, \sigma=\sigma_{n}^i\,|\,\Yb(0)=\Yb_0\}.\label{eq:totprob}
\end{align}
That is we will calculate the probability of the event $\{A\text{ and }1\leq K\leq m\}$ under the restriction that the first separation happens at the $n^{\text{th}}$ transition that generates an $i$-gap for each possible width $j\in\Zbb^+$ at separation and for every $i\in\Zbb^+$ and $n\in\Zbb^+_0$.

For an $i\in\Zbb^+$ and $t\geq0$, let $B_i(t)$ be the event that $\sum_{l=i_{\lrm}}^{i-1}Y_{l}(\sigma)$ and $\sum_{l=i+1}^{i_\rrm}Y_{l}(\sigma)$ are both \emp{even} numbers. Due to the parity conserving property of double branching annihilating random walkers we can simplify the previous display, namely
\begin{align*}
\eqref{eq:totprob}=\sum_{j=5}^{+\infty}\sum_{i=1}^{j-2}\sum_{n=1}^{+\infty}\Prob\{A, 1\leq K\leq m, B_i(\sigma), W(\sigma)=j, \sigma=\sigma_{n}^i\,|\,\Yb(0)=\Yb_0\}.
\end{align*}
The strong Markov property then implies that
\begin{align*}
&\Prob\{A\text{ and }1\leq K\leq m\,|\,\Yb(0)=\Yb_0\}\\
&\quad\leq \sum_{j=5}^{+\infty}\sum_{i=1}^{j-2}\sum_{n=1}^{+\infty}\sum_{\yb\in\Sc}
\Prob\{\Yb(\sigma_{n}^i)=\yb, B_i(\sigma_{n}^i), W(\sigma_{n}^i)=j, \sigma_{n}^i<\tau_N\,|\,\Yb(0)=\Yb_0\}\\
&\qquad\qquad\qquad\qquad\;\cdot\Prob\{A_i, 0\leq K < m\,|\,\Yb(\sigma_{n}^i)=\yb, W(\sigma_{n}^i)=j, B_i(\sigma_{n}^i), \sigma_{n}^i<\tau_N\},
\end{align*}
since $\sigma<\tau_N$ whenever $K\geq 1$. $A_i$ is the event that an $i$-gap separates the process $\Yb$ into two sub-processes which remain separated by at least one vacant site and each dies out in finite time.

Fix an $n\in\Zbb^+_0$ and let $i,j\in\Zbb^+$ be those indices for which the configuration $\yb\in\Sc$ has an $i$-gap and $w(\yb)=j$. Let $(\tilde{\Yb}^{(1)}(t))_{t\geq0}$ and $(\tilde{\Yb}^{(2)}(t))_{t\geq0}$ be two independent double branching annihilating random walks with initial configurations $\tilde{\yb}^{(1)}\in\Sc$ and $\tilde{\yb}^{(2)}\in\Sc$ for which $\tilde{y}^{(1)}_{l}=y_l\ind\{l\leq i\}$ and $\tilde{y}^{(2)}_l=y_l\ind\{l>i\}$, respectively, for all $l\in\Zbb$. We claim that
\begin{multline}
\Prob\{A_i \text{ and } 0\leq K\leq m-1\,|\,\Yb(\sigma_{n}^i)=\yb, W(\sigma_{n}^i)=j, B_i(\sigma_{n}^i),\sigma_n^i<\tau_N\}\\
\leq \Prob\{A^{(1)},A^{(2)}\} = \Prob\{A^{(1)}\}\cdot\Prob\{A^{(2)}\}\label{eq:jointdeath},
\end{multline}
where $A^{(1)}$ ($A^{(2)}$) is the event that the process $\tilde{\Yb}^{(1)}$ ($\tilde{\Yb}^{(2)}$) dies out with at most $m-1$ separation times.
To prove the above inequality we realize $\tilde{\Yb}^{(1)}$ and $\tilde{\Yb}^{(2)}$ in a common probability space giving birth to the joint process $\tilde{\Yb}$. This evolves according to the following rules.
\begin{itemize}
\item Initially, with probability $\frac{q^{(a)}}{q^{(1)}+q^{(2)}}$, we execute the first transition of $\tilde{\Yb}^{(a)}$, where $q^{(a)}$ is the sum of the rates of the particles of $\tilde{\yb}^{(a)}$ for $a\in\{1,2\}$.
\item After any transition being executed we choose the next with probability
\[
\frac{q^{(a)}(t)}{q^{(1)}(t)+q^{(2)}(t)}
\]
at which we execute the subsequent transition of $\tilde{\Yb}^{(a)}(t)$ (that is not used yet until $t$), where $q^{(a)}(t)$ is the sum of the rates of the particles of $\tilde{\Yb}^{(a)}(t)$ for $a\in\{1,2\}$.
\end{itemize}
We apply the previous rules \emp{until} a particle of $\tilde{\Yb}^{(1)}$ does not meet, i.e.\ become adjacent, with a particle of $\tilde{\Yb}^{(2)}$. When a meeting occurs we clear all subsequent events of $\tilde{\Yb}^{(1)}$ and $\tilde{\Yb}^{(2)}$. From that time the paths of $\tilde{\Yb}$ are built up from a completely independent new set of clocks.

It is not hard to see that $\Yb$, restarted from $\sigma_n^i$ along with the conditions that $\Yb(\sigma_n^i)=\yb$, $B_i(\sigma_{n}^i)$, $W(\sigma_{n}^i)=j$, $\sigma_{n}^i<\tau_N$, and the process $\tilde{\Yb}$ share the same probability law. Furthermore, under the restriction that $\tilde{\Yb}$ initially falls apart into two sub-processes by the $i$-gap and they remain separated by at least one vacant site, the event that $\tilde{\Yb}$ dies out with at most $m-1$ separation times \emp{implies} the occurrence of both of the events $A^{(1)}$ and $A^{(2)}$ by the coupling. This eventually proves the inequality of \eqref{eq:jointdeath}.

Now, applying the induction hypothesis separately for $\Prob\{A^{(1)}\}$ and $\Prob\{A^{(2)}\}$ of \eqref{eq:jointdeath}, we obtain that
\begin{multline*}
\Prob\{A_i \text{ and } 0\leq K\leq m-1\,|\,\Yb(\sigma_{n}^i)=\yb, W(\sigma_{n}^i)=j, B_i(\sigma_{n}^i),\sigma_n^i<\tau_N\}\\
\leq \bigg(C(b)\cdot \frac{2^{-i}}{i^2}\bigg)\cdot \bigg(C(b)\cdot \frac{2^{-(j-i-1)}}{(j-i-1)^2}\bigg)
\leq 18C(b)^2\cdot\frac{2^{-j}}{i^2\cdot (j-i)^2}.
\end{multline*}
It then follows that
\begin{multline*}
\Prob\{A\text{ and }1\leq K\leq m\,|\,\Yb(0)=\Yb_0\}\\
\leq 18C(b)^2\cdot\sum_{j=5}^{+\infty}\sum_{i=1}^{j-2}\frac{2^{-j}}{i^2\cdot (j-i)^2}\cdot
\sum_{n=1}^{+\infty}\Prob\{\sigma_{n}^i<\tau_N, W(\sigma_{n}^i)=j\,|\,\Yb(0)=\Yb_0\}.
\end{multline*}
It is only left to estimate the inner sum of the previous display. This can be achieved by using Proposition \ref{prop:widthgrowth} and the same large deviation estimates we established in case of $k=0$. We define $\vartheta$ to be the following stopping time. If $j < w(\Yb_0)$ then $\vartheta=0$, otherwise, in case of $j\geq w(\Yb_0)$, $\vartheta$ is the (possibly) random index when $(W(\tau_l))_{l\in\Zbb_0^{+}}$ first reaches $j$.

Let $M^{(j)}$ be the total number of times the process $(W(\tau_l)-w(\Yb_0))_{\vartheta \leq l <\tau_N}$ hits $j-w(\Yb_0)$. Then
\begin{align*}
&\sum_{n=1}^{+\infty}\Prob\{\sigma_{n}^i<\tau_N, W(\sigma_{n}^i)=j\,|\,\Yb(0)=\Yb_0\}\\
&\quad=\Exp\bigg[\sum_{n=1}^{+\infty}\ind\{\sigma_{n}^i<\tau_N, W(\sigma_{n}^i)=j\} \,\bigg|\, \Yb(0)=\Yb_0\bigg]\\
&\quad\leq 2\cdot \Exp\big[M^{(j)} \,\big|\, \Yb(0)=\Yb_0\big]
=2\cdot \sum_{n=1}^{+\infty}\Prob\{M^{(j)}\geq n \,\big|\, \Yb(0)=\Yb_0\},
\end{align*}
where the extra factor in front of the expectation comes from the fact that any lattice point can be vacated at most with rate $2$ provided that $b>10^4$ and $p<2b^{-2}$. Now, the same argument we used in establishing \eqref{eq:largedevestim} can be pushed through here as well. By replacing the corresponding constants we obtain that
\[
\Prob\{M^{(j)}\geq n \,\big|\,\Yb(0)=\Yb_0\} =
\left\{
\begin{array}{ll}
4^{j-w(\Yb_0)}\cdot 2^{1-n}, &\quad \hbox{if $j < w(\Yb_0)$;}\\
4\cdot 2^{1-n}, &\quad \hbox{if $j \geq w(\Yb_0)$}
\end{array}
\right.
\]
using the strong Markov property for the stopping time $\vartheta$ and restarting $\Yb$ from that time. Putting the pieces together we arrive to
\begin{align}
&\Prob\{A\text{ and }1\leq K\leq m\,|\,\Yb(0)=\Yb_0\}\nonumber\\
&\qquad\leq
144C(b)^2\cdot 2^{-w(\Yb_0)}\cdot \sum_{j=5}^{w(\Yb_0)-1}\frac{1}{(w(\Yb_0)-j)^2}\cdot\sum_{i=1}^{j-2}\frac{1}{i^2\cdot (j-i)^2}\nonumber\\
&\qquad
+288C(b)^2\cdot\sum_{j=w(\Yb_0)}^{+\infty}2^{-j}\cdot\sum_{i=1}^{j-2} \frac{1}{i^2\cdot (j-i)^2},\label{eq:lastbutone}
\end{align}
It is not hard to see that for any integer $u\geq 2$:
\[
\sum_{x=1}^{u-1}\frac{1}{x^2(u-x)^2} = \frac{2}{u^2}\bigg[\frac{2}{u}\sum_{x=1}^{u-1}\frac{1}{x}+\sum_{x=1}^{+\infty}\frac{1}{x^2}\bigg] \leq \frac{6}{u^2}.
\]
Applying the previous estimate for \eqref{eq:lastbutone} we obtain that
\begin{align*}
\Prob\{A\text{ and }1\leq K\leq m\,|\,\Yb(0)=\Yb_0\}
&\leq 10^4\cdot C(b)^2\cdot \frac{2^{-w(\Yb_0)}}{w(\Yb_0)^2}\\
&\leq \bigg(C(b)-\frac{128}{2+b^2}\bigg)\cdot \frac{2^{-w(\Yb_0)}}{w(\Yb_0)^2},
\end{align*}
since $b>10^4$, which completes the proof.
\end{proof}

\addcontentsline{toc}{section}{References}%
\bibliographystyle{plain}%
\bibliography{survdbarw}%

\end{document}